\documentclass[preprint,12pt]{elsarticle}

\usepackage{amssymb}
\usepackage{amsthm}

\newcommand{\sysn}{\left\{\begin{array}{rcl}}
\newcommand{\sysk}{\end{array}\right.}

\newtheorem{theorem}{Theorem}[section]
\newtheorem{lemma}[theorem]{Lemma}

\theoremstyle{example}

\theoremstyle{definition}
\newtheorem{definition}[theorem]{Definition}
\newtheorem{remark}[theorem]{Remark}
\newtheorem{corollary}[theorem]{Corollary}

\journal{...}

\begin{document}

\begin{frontmatter}



\title{On the cardinality of $S(n)$-spaces}


\author{Alexander V. Osipov}

\ead{OAB@list.ru}


\address{Krasovskii Institute of Mathematics and Mechanics, Ural Federal
 University,

 Ural State University of Economics, Yekaterinburg, Russia}

\begin{abstract}

In this paper, for a topological space $X$ and any positive
integer $n$, we define the cardinal functions $sL_{\theta(n)}(X)$,
$\theta(n)$-quasi-Menger number $qM_{\theta(n)}(X)$ and
$s(n)$-quasi-Menger number $qM_{s(n)}(X)$. We prove the following
statements:

$\bullet$ For every $S(2n)$-space $X$,  $|X|\leq 2^{sL_{\theta
(n)}(X)\kappa_{\theta (n)}(X)}$.

$\bullet$ For every $S(2n)$-space $X$, $|X|\leq 2^{qM_{\theta
(n)}(X)\kappa_{\theta (n)}(X)}$.

$\bullet$ For every $S(2n)$-space $X$, $|X|\leq 2^{qM_{s
(n)}(X)\kappa_{\theta (n)}(X)}$.

Similar results are stated for $S(2n-1)$-spaces.



\end{abstract}


\end{frontmatter}



\section{Introduction}

 In 1966 Velichko \cite{vel} introduced the notion of $\theta$-closedness.
 For a subset $M$ of a topological space $X$ the $\theta$-closure
 is defined by $cl_{\theta} M=\{ x\in X:$ every closed
 neighborhood of $x$ meets $M \}$, $M$ is $\theta$-closed if
 $cl_{\theta} M=M$. This concept was used by many authors for the
 study of Hausdorff nonregular spaces \cite{digi,gk,ham,her,povo,str} and [16-23]. The $S(n)$-spaces
 were introduced by Viglino in 1969 (see \cite{vig}) under the name
 $\overline{T}_n$-spaces. After that $S(n)$-spaces, $S(n)$-closed
 and $S(n)$-minimal spaces were studied by other authors. For
 example, Porter in 1969 (see \cite{por}) studied minimal
 $R(\omega_0)$ spaces, where he used the notation $R(n)$ for
 $S(2n-1)$-spaces and $U(n)$ for $S(2n)$-spaces. For the first
 time the notation $S(n)$ for $S(n)$-spaces appeared in 1973 in
 \cite{povo} where the authors extended the definition of
 $S(n)$-spaces to $S(\alpha)$-spaces, where $\alpha$ is any
 ordinal. In that paper Porter and Votaw, among other results,
 characterized the minimal $S(\alpha)$ and $S(\alpha)$-closed
 spaces. In paper \cite{digi} Dikranjan and Giuli  introduced more general notion
$\theta^n$-closure operator and characterized the
$S(n)$-$\theta$-closed spaces.

 In this
paper we continue to study properties of $S(n)$-spaces by applying
the notions of $\theta^n$-closure operators.

In Section 3 we introduce new cardinal functions:
$sL_{\theta(n)}(X)$, $\theta(n)$-quasi-Menger number
$qM_{\theta(n)}(X)$ and $s(n)$-quasi-Menger number $qM_{s(n)}(X)$
in order to extend some known cardinality bounds for Hausdorff and
Urysohn spaces in the case of $S(n)$-spaces [5-8]. In particular
we prove the following:

$\bullet$ For every $S(2n)$-space $X$,  $|X|\leq 2^{sL_{\theta
(n)}(X)\kappa_{\theta (n)}(X)}$ (Theorem \ref{th1}). For $n=1$ we
have Theorem 1 in \cite{alko}.

$\bullet$ For every $S(2n)$-space $X$, $|X|\leq 2^{qM_{\theta
(n)}(X)\kappa_{\theta (n)}(X)}$ (Theorem \ref{th3}). For $n=1$ we
have Theorem 3 in \cite{alko}.

$\bullet$ For every $S(2n)$-space $X$, $|X|\leq 2^{qM_{s
(n)}(X)\kappa_{\theta (n)}(X)}$ (Theorem \ref{th5}).

\section{Main definitions and notation}
\begin{definition} (\cite{digi}). Suppose that $X$ is a topological space,
$M\subset X$, and $x\in X$. For each $n\in \mathbb{N}$, the {\it
$\theta^n$-closure operator} is defined as follows: $x\notin
cl_{\theta^n} M$ if there exists a set of open neighborhoods
$U_1\subset U_2\subset ... U_n$ of the point $x$ such that $cl U_i
\subset U_{i+1}$ for $i =1,2,..., n-1$ and $cl U_n\cap
M=\emptyset$. For $n=0$, we put $cl_{\theta^0} M=cl M$.
\end{definition}

For $n=1$, this definition gives the $\theta$-closure operator
defined by Velichko (\cite{vel}).

A set $M$ is said to be $\theta^n$-closed if $M=cl_{\theta^n}M$.
Similarly the {\it $\theta^n$-interior} of $M$ is defined and
denoted by $Int_{\theta^n} M$, so $Int_{\theta^n}M=X\setminus
cl_{\theta^n}(X\setminus M)$.

For any $n\in \mathbb{N}$, a point $x\in X$ is {\it
$S(n)$-separated} from a subset $M$ if $x\notin cl_{\theta^n}M$.
For example, $x$ is $S(0)$-separated from $M$ if $x\notin
\overline{M}$. For $n>0$ the relation 'being $S(n)$-separated'
between points is symmetric. On the other hand 'being
$S(0)$-separated' can be highly nonsymmetric in non- $T_1$ spaces.
This is why we say that two points $x$ and $y$ are
$S(0)$-separated if $x\notin \{\overline{y}\}$ and $y\notin
\{\overline{x}\}$. Since we are going to consider here only
$T_1$-spaces, for us the $S(0)$-spaces will be exactly the
$T_1$-spaces.

\begin{definition} (\cite{digi}). Let $n$ be a positive integer and $X$ be a space.

$\bullet$ $X$ is an {\it $S(n)$-space} (or $X$ satisfies the {\it
$S(n)$ separation axiom}) if any two different points in $X$ are
$S(n)$-separated;

$\bullet$ an open cover $\{U_{\alpha}\}$ of $X$ is an $S(n)$-cover
if every point of $X$ is in the $\theta^n$-interior of some
$U_{\alpha}$.

\end{definition}

The $S(n)$-spaces coincide with the $\overline{T}_n$-spaces
defined in \cite{vig} and studied further in \cite{povo}, where
$S(\alpha)$-spaces are defined for each ordinal $\alpha$ (see also
\cite{pr2}). The open covers with an $n-1$ chain of shrinkable
refinements ($S(n)$-cover in Dikranjan and Giuli's terminology)
were defined in \cite{povo}.

Obviously, any $S(0)$-space is $T_0$, any $S(1)$-space is
Hausdorff, and any $S(2)$-space is Urysohn.

  In the class of topological $S(n)$-spaces, $S(n)$-closed ($S(n)$-$\theta$-closed)
spaces are defined as $S(n)$-spaces which are closed
(respectively, $\theta$-closed) in any ambient $S(n)$-space.

\begin{definition} (\cite{os2}).  An open set $U$ is called an {\it $n$-hull}  of a
set $A$  if there exists a family of open sets $U_1$, $U_2,...,
U_n=U$ such that $A\subseteq U_1$ and $cl U_i\subseteq U_{i+1}$
for $i=1,...,n-1$.
\end{definition}

By a closed $n$-hull of a set $A$ we mean the closure of any
$n$-hull of $A$.

\bigskip

All necessary definitions related to $S(n)$-spaces can be found in
\cite{digi,os2,os4,os5,povo,str}.

\section{On cardinality bounds for $S(n)$-spaces}

 A.V. Arhangel'skii  introduced the relative cardinal function $sl(Y,X)$ (see \cite{arh}, p.324).
When $Y=X$ we have $sl(X,X)=wL_c(X)$, where $wL_c(X)$ is the weak
Lindel$\ddot{o}$f degree of $X$ with respect to closed sets. The
name and the notation $wL_c(X)$ come from a paper by O. Alas from
1993 (see \cite{al}) but the same cardinal function
(quasi-Lindel$\ddot{o}$f degree of $X$) denoted by $ql(X)$ was
studied by Arhangel'skii in 1979 (\cite{arh79}).

The weak Lindel$\ddot{o}$f degree of $X$ with respect to closed
sets (quasi-Lindel$\ddot{o}$f degree of $X$), denoted below as
$wL_c(X)$, is the smallest infinite cardinal $\kappa$ such that
for every closed subset $H$ of $X$ and every collection
$\mathcal{V}$ of open sets in $X$ that covers $H$, there is a
subcollection $\mathcal{V}_0$ of $\mathcal{V}$ such that
$|\mathcal{V}_0|\leq \kappa$ and $H\subseteq \overline{\bigcup
\mathcal{V}_0}$. In 1979, Arhangel'skii showed that if $X$ is a
regular space, then $|X|\leq 2^{\chi(X)wL_c(X)}$ (see
\cite{arh79,hod}). Later, in 1993, Alas extended this result to
the class of Urysohn spaces (\cite{al,hod}).

 O.T. Alas and Lj.D.R. Ko$\check{c}$inac introduced the following definition.

\begin{definition}(\cite{alko}) For a space $X$, $sL_{\theta}(X)$
is the smallest cardinal $\kappa$ such that if $A\subset X$,
$\mathcal{U}$ is an open collection and
$cl_{\theta}(A)\subset\bigcup \mathcal{U}$, then there is
$\mathcal{V}\subset \mathcal{U}$ with $|\mathcal{V}|\leq \kappa$
and $A\subset\overline{\bigcup \mathcal{V}}$.
\end{definition}

It is immediate that $sL_{\theta}(X)\leq sL(X)$ for every space
$X$.

\begin{definition}(\cite{alko}) For a Hausdorff space $X$, let $\kappa(X)$
be the smallest cardinal $\kappa$ such that for each point $x\in
X$, there is a collection $\mathcal{V}_{x}$ of closed
neighborhoods of $x$ so that $|\mathcal{V}_{x}|\leq \kappa$ and if
$W$ is a closed neighborhood of $x$, then $W$ contains a member of
$\mathcal{V}_{x}$.
\end{definition}

Note that $\kappa(X)\leq\chi(X)$ and $\kappa(X)=\chi(X_s)$ where
$X_s$ is the semiregularization of $X$.

It was shown in (\cite{alko}, Theorem 1) that  $|X|\leq
     2^{sL_{\theta}(X)\kappa(X)}$ whenever $X$ is a Urysohn
space.

We introduce the following definitions.

\begin{definition} For a space $X$ and $n\in \mathbb{N}$, $sL_{\theta(n)}(X)$
is the smallest cardinal $\kappa$ such that if $A\subset X$,
$\mathcal{U}$ is an open collection and
$cl_{\theta^n}(A)\subset\bigcup \mathcal{U}$, then there is
$\mathcal{V}\subset \mathcal{U}$ with $|\mathcal{V}|\leq \kappa$
and $A\subset\overline{\bigcup \mathcal{V}}$.
\end{definition}

It is immediate that $sL_{\theta(n)}(X)\leq
sL_{\theta(n-1)}(X)\leq ...\leq sL_{\theta}(X)\leq sL(X)$ for
every space $X$.

\begin{definition} For a space $X$ and $n\in \mathbb{N}$, let $\kappa_{\theta(n)}(X)$ be the smallest cardinal $\kappa$ such
that for each point $x\in X$, there is a collection
$\mathcal{V}_{x}$ of closed $n$-hulls of $x$ so that
$|\mathcal{V}_{x}|\leq \kappa$ and if $W$ is a closed $n$-hull of
$x$, then $W$ contains a member of $\mathcal{V}_{x}$.
\end{definition}

We need the following lemma.

\begin{lemma}\label{lem1} For a subset $A$ of an $S(2n)$-space $X$, $|cl_{\theta
^n}{A}|\leq|A|^{\kappa_{\theta (n)}(X)}$.
\end{lemma}

\begin{proof} Let $\kappa=\kappa_{\theta(n)}(X)$ and $\mu=|A|$.
Consider for each point $x\in X$, a collection $\mathcal{V}_{x}$
of closed $n$-hulls of $x$ so that $|\mathcal{V}_{x}|\leq \kappa$
and if $W$ is a closed $n$-hull of $x$, then $W$ contains a member
of $\mathcal{V}_{x}$.

For each  point $x\in cl_{\theta ^n}{A}$ and each $V\in
\mathcal{V}_{x}$ we fix a point $x_V\in A\cap V$. Consider the set
$A_x:=\bigcup\{x_V : V\in \mathcal{V}_{x}\}$. It is clear that
$|A_x|\leq \kappa$ and $A_x\subseteq A$ for each $x\in cl_{\theta
^n}{A}$. Let $\Gamma_x$ be the family $\{V\cap A_x :
V\in\mathcal{V}_{x}\}$. Since $x\in V\cap cl_{\theta ^n}{A_x}$ and
moreover $\bigcap\limits_{V\in \mathcal{V}_{x}} V\cap cl_{\theta
^n}{A_x}\subset \bigcap\limits_{V\in \mathcal{V}_{x}} V\subset
\{x\}$ by the fact that $X$ is an $S(2n)$-space, we have
$\bigcap\limits_{V\in \mathcal{V}_{x}} V\cap cl_{\theta
^n}{A_x}=\{x\}$. This implies that the correspondence $x
\rightarrow \Gamma_x$ defines a one to one map from $cl_{\theta
^n}{A}$ into $exp_{\kappa}(exp_{\kappa}(A))$. As
$|exp_{\kappa}(exp_{\kappa}(A))|\leq
(\mu^{\kappa})^{\kappa}=\mu^{\kappa}$ we have $|cl_{\theta
^n}{A}|\leq \mu^{\kappa}=|A|^{\kappa_{\theta (n)}(X)}$.

\end{proof}

\begin{theorem}\label{th1} For every $S(2n)$-space $X$,
$|X|\leq 2^{sL_{\theta (n)}(X)\kappa_{\theta (n)}(X)}$.

\end{theorem}

\begin{proof} Applying the well-known method of
Pol-$\check{S}$apirovskii-Arhangel'skii-Grizlov \cite{arh,gr,pol},
let $\tau=sL_{\theta (n)}(X)\kappa_{\theta
 (n)}(X)$ and for each $x\in X$  let $\mathcal{B}_{x}$ be a collection of closed
 $n$-hulls
 of $x$ such that  $|\mathcal{B}_{x}|\leq\tau$ and every closed $n$-hull $W$ of $x$ contains a member of $\mathcal{B}_{x}$.

We shall define an increasing sequence
      $\{A_{\alpha}:\alpha\in\tau^{+}\}$ of subsets of $X$ and a
      sequence  $\{\mathcal{U}_{\alpha}:\alpha\in\tau^{+}\}$ of collections
      of open subsets of $X$ such that:

     (1) $|A_{\alpha}|\leq2^{\tau}$ for every $\alpha<\tau^{+}$  and $A_{\alpha}\supset
           cl_{\theta^n}(\bigcup_{\beta<\alpha}A_{\beta})$  for every
           $\alpha<\tau^{+}$;

     (2) $\mathcal{U}_{\alpha}=\{Int(M):M\in\bigcup\{\mathcal{B}_{x}:x\in
            cl_{\theta^n}(\bigcup_{\beta<\alpha}A_{\beta})\}\}$ for every
             $\alpha<\tau^{+}$;

     (3) If $\mathcal{V}\in\lbrack \mathcal{U}_{\alpha}\rbrack ^{\leq\tau}$
     and $\overline{\bigcup \mathcal{V}}\ne
   X$, then  $A_{\alpha+1}\setminus\overline{\bigcup
   \mathcal{V}}\ne\emptyset$ for every
             $\alpha<\tau^{+}$.

\bigskip
Suppose that the sets  $A_{\beta}$ and $\mathcal{U}_{\beta}$,
 satisfying (1)-(3), have been defined for all
 $\beta<\alpha<\tau^+$ and let us define $A_{\alpha}$ and
 $\mathcal{U}_{\alpha}$.

 Note that it follows from Lemma \ref{lem1} that $|cl_{\theta
^n}(\bigcup_{\beta<\alpha}A_{\beta})|\leq|(\bigcup_{\beta<\alpha}A_{\beta})|^{\kappa_{\theta
(n)}(X)}$ and therefore
$|cl_{\theta^n}(\bigcup_{\beta<\alpha}A_{\beta})|\leq 2^{\tau}$.

 For each $\mathcal{V}\in\lbrack
\mathcal{U}_{\alpha}\rbrack ^{\leq\tau}$ such that
$X\setminus\overline{\bigcup \mathcal{V}}\ne\emptyset$, fix a
point $x_{\mathcal{V}}\in X\setminus\overline{\bigcup
\mathcal{V}}$. Let $A_{\alpha}=cl_{\theta^n}(\{x_{\mathcal{V}}:
\mathcal{V}\in [\mathcal{U}_{\alpha}]^{\leq \tau}$, $X\setminus
\overline{\bigcup \mathcal{V}}\neq \emptyset\}\cup
\bigcup_{\beta<\alpha}A_{\beta})$. Then $|A_{\alpha}|\leq
2^{\tau}$.

Finally, let  $A=\bigcup\{A_{\alpha}:\alpha<\tau^{+}\}$. Then
$cl_{\theta^n}(A)=A$. Indeed, let $y\in X$ so that the closure of
each $n$-hull of $y$ intersects $A$; then for each $F\in
\mathcal{B}_{y}$ there is  $\alpha_{F}\leq\tau^+$ so that $F\cap
A_{\alpha_{F}}\ne\emptyset$. Since $|\{\alpha_{F}:F\in
\mathcal{B}_{y}\}|\leq\tau$, there is  $\psi<\tau^+$, so that
$\psi>\alpha_{F}$ for every $F\in \mathcal{B}_{y}$ and  $y\in
cl_{\theta^n}(A_{\psi})\subseteq A$.

Now it is enough to show that  $A=X$. On the contrary, assume that
there is $y\in X\setminus A$. Since $cl_{\theta^n}(A)=A$ there is
$W\in \mathcal{B}_{y}$ so that $W\cap A=\emptyset$. Since $W$ is a
closed $n$-hull of $y$, there exists a family of open sets $U_1$,
$U_2,..., U_n$ such that $y\in U_1$, $\overline{U_i}\subseteq
U_{i+1}$ for $i=1,...,n-1$ and $W=\overline{U_n}$. Let
$V_1=X\setminus W$, $V_2=X\setminus \overline{U_{n-1}}$, ...,
$V_n=X\setminus \overline{U_1}$. Then $V_n$ is an open $n$-hull of
$A$ (hence it is an open $n$-hull for each point $x\in A$) and
$\overline{V_n}\cap U_1=\emptyset$. For each $x\in A$ choose a
closed $n$-hull $D_{x}\in \mathcal{B}_{x}$ so that $D_{x}\subset
\overline{V_n}$.

 Since
$\{Int(D_{x}):x\in A\}$ is an open cover of $A=cl_{\theta^n}(A)$,
there is $B\subset A$, so that $|B|\leq sL_{\theta
(n)}(X)\leq\tau$ and $A\subset\overline{\bigcup\limits_{x\in
B}Int(D_{x})}$. Since $|B|\leq\tau$, there is $\beta<\tau^+$ so
that $B\subset A_{\beta}$ and $\mathcal{V}=\{Int(D_{x}):x\in B\}$
is a convenient collection of open sets with cardinality $\leq
\tau$ which appears at the step $\beta+1$. Hence,
$A_{\beta+1}\setminus\overline{\bigcup \mathcal{V}}\ne\emptyset$
and we have a contradiction with
$A\subset\overline{\bigcup\limits_{x\in B}Int(D_{x})}$.

\end{proof}

\begin{corollary}\label{th101} (Theorem 1 in \cite{alko}) For every Urysohn space $X$,
$|X|\leq 2^{sL_{\theta}(X)\kappa_{\theta}(X)}$.

\end{corollary}

\begin{definition} (see \cite{alko} for $n=1$) For a space $X$, the $\theta(n)$-quasi-Menger
number $qM_{\theta(n)}(X)$ is the smallest cardinal number
$\kappa$ such that for every closed subset $A$ of $X$ and every
collection $\{\mathcal{U}_{\alpha}: \alpha\leq \kappa \}$ of
families of open subsets of $X$ with $A\subset
\bigcup_{\alpha<\kappa} (\bigcup \mathcal{U}_{\alpha})$, there are
finite subfamilies $\mathcal{V}_{\alpha}$ of
$\mathcal{U}_{\alpha}$, $\alpha<\kappa$, such that $A\subset
\bigcup_{\alpha<\kappa} cl_{\theta^n}(\bigcup
\mathcal{V}_{\alpha})$.

\end{definition}

It is immediate that $qM_{\theta(n)}(X)\leq
qM_{\theta(n-1)}(X)\leq ...\leq qM_{\theta(1)}(X)=qM_{\theta}(X)$
for every space $X$.

\begin{theorem}\label{th3} For every $S(2n)$-space $X$, $|X|\leq 2^{qM_{\theta (n)}(X)\kappa_{\theta
(n)}(X)}$.
\end{theorem}

\begin{proof} Let $\kappa=qM_{\theta (n)}(X)\kappa_{\theta
(n)}(X)$ and for each $x\in X$ let $\mathcal{B}_{x}$ be a
collection of closed $n$-hulls of $x$ such that
$|\mathcal{B}_{x}|\leq \kappa$ and every closed $n$-hull of $x$
contains a member of $\mathcal{B}_{x}$. We shall define an
increasing sequence $\{F_{\alpha}: \alpha\in \kappa^+\}$ of
subsets of $X$ and a sequence $\{\mathcal{U}_{\alpha}: \alpha\in
\kappa^+\}$ of collections of open subsets of $X$ satisfying the
following conditions:

     (1) $|F_{\alpha}|\leq2^{\kappa}$ for every  $\alpha<\kappa^{+}$;

     (2) $\mathcal{U}_{\alpha}=\{Int(M):M\in\bigcup\{\mathcal{B}_{x}: x\in
            cl_{\theta^n}(\bigcup_{\beta<\alpha}F_{\beta})\}\}$ for every
             $\alpha<\kappa^{+}$;

     (3) If $\mathcal{V}\in\lbrack \mathcal{U}_{\alpha}\rbrack ^{\leq\kappa}$
     and $\overline{\bigcup \mathcal{V}}\ne
   X$, then  $F_{\alpha}\setminus \overline{\bigcup
   \mathcal{V}}\ne\emptyset$ for every
             $\alpha<\kappa^{+}$.

Suppose $\alpha<\kappa^+$ and the sets $F_{\beta}$ and
$\mathcal{U}_{\beta}$ satisfying (1)-(3) are already defined for
all $\beta<\alpha$. Let us define $F_{\alpha}$ and
$\mathcal{U}_{\alpha}$.

Put $M_{\alpha}=cl_{\theta^n}(\bigcup_{\beta<\alpha} F_{\beta})$.
By Lemma \ref{lem1}, $|M_{\alpha}|\leq 2^{\kappa}$, hence,
$|\mathcal{U}_{\alpha}|\leq 2^{\kappa}$. For every $\mathcal{V}\in
[\mathcal{U}_{\alpha}]^{\leq \kappa}$ such that $\overline{\bigcup
\mathcal{V}}\ne X$ take a point $x_{\mathcal{V}}\in X\setminus
\overline{\bigcup \mathcal{V}}$ and define
$F_{\alpha}:=cl_{\theta^n}(\{x_{\mathcal{V}}: \mathcal{V}\in
[\mathcal{U}_{\alpha}]^{\leq \kappa}$, $X\setminus
\overline{\bigcup \mathcal{V}}\neq \emptyset\}\cup
\bigcup_{\beta<\alpha} F_{\beta})$. Then $|F_{\alpha}|\leq
2^{\kappa}$.

 Let $F=\bigcup \{F_{\alpha} : \alpha<\kappa^+\}$. Then $|F|\leq
 2^{\kappa}$. We claim that $cl_{\theta^n}(F)=X$.
 First, we show that  $cl_{\theta^n}(F)=\bigcup_{\alpha<\kappa^+}
 cl_{\theta^n}(F_{\alpha})$.  Let $x\in cl_{\theta^n}(F)$. The
 closure of every $n$-hull of $x$ intersects $F$, so that for each
 $B\in B_{x}$ one can find some $\alpha_B<\kappa^+$ for which
 $B\cap F_{\alpha_B}\ne\emptyset$. Since $\kappa^+$ is a
 regular cardinal and $|\{\alpha_B : B\in \mathcal{B}_{x}\}|\leq
 \kappa$, there exists $\beta<\kappa^+$ such that $\beta>\alpha_B$
 for every $B\in \mathcal{B}_{x}$ and $x\in
 cl_{\theta^n}(F_{\beta})\subseteq F$.

 We claim that $cl_{\theta^n}(F)=F$. Suppose there is $y\in X\setminus cl_{\theta^n}(F)$. Let
 $\mathcal{B}_{y}=\{B_{y}({\alpha}): \alpha<\kappa \}$.
 For each $\alpha<\kappa$ let $\mathcal{W}_{\alpha}$ be the
 collection of all members $W\in \bigcup\{\mathcal{B}_{x} : x\in
 cl_{\theta^n}(F)\}$ such that $B_{y}({\alpha})\cap
 W=\emptyset$. Since $X$ is an $S(2n)$-space,
 $cl_{\theta^n}(F)\subset \bigcup_{\alpha<\kappa} \bigcup\{Int(W):
 W\in \mathcal{W}_{\alpha}\}$. As $cl_{\theta^n}(F)$ is $\theta^n$-closed,
 one can choose $\mathcal{V}_{\alpha}\in
 [\mathcal{W}_{\alpha}]^{<\omega}$ for each $\alpha\in \kappa$
 such that $cl_{\theta^n}(F)\subset \bigcup_{\alpha<\kappa}
 cl_{\theta^n}(\bigcup \{Int(V) : V\in \mathcal{V}_{\alpha}\})$.
 Clearly, for every $\alpha< \kappa$, $\bigcup \{Int(V): V\in
 \mathcal{V}_{\alpha}\}\cap B_{y}({\alpha})=\emptyset$,
 hence, $y\notin cl_{\theta^n}(\bigcup\{Int(V): V\in
 \mathcal{V}_{\alpha}\})$. This means $y\notin \bigcup_{\alpha<\kappa} cl_{\theta^n}(\bigcup\{Int(V): V\in
 \mathcal{V}_{\alpha}\})$. There is a $\beta<\kappa^+$ such that
 all $\mathcal{V}_{\alpha}$, $\alpha<\kappa$, are contained in
 $\mathcal{U}_{\beta}$. Then by (3), $F_{\beta+1}\setminus \bigcup_{\alpha<\kappa} cl_{\theta^n}(\bigcup\{Int(V): V\in
 \mathcal{V}_{\alpha}\})\ne \emptyset$ which is a contradiction.

\end{proof}

\begin{corollary}\label{th301} $($Theorem 3 in \cite{alko}$)$ For every Urysohn space $X$, $|X|\leq 2^{qM_{\theta}(X)\kappa_{\theta
}(X)}$.
\end{corollary}

\begin{remark}
Similarly, we can define the $\theta_0^n$-closure operator and
obtain similar cardinality bounds for $S(2n-1)$-spaces.
\end{remark}

\begin{definition} Suppose that $X$ is a topological space,
$M\subset X$, and $x\in X$. For each $n\in \mathbb{N}$, the {\it
$\theta_0^n$-closure operator} is defined as follows: $x\notin
cl_{\theta_0^n} M$ if there exists a set of open neighborhoods
$U_1\subset U_2\subset ... U_n$ of the point $x$ such that $cl U_i
\subset U_{i+1}$ for $i =1,2,..., n-1$ and $U_n\bigcap
M=\emptyset$.
\end{definition}

\begin{definition} For a space $X$ and $n\in \mathbb{N}$, $sL_{\theta_0(n)}(X)$
is the smallest cardinal $\kappa$ such that if $A$ is a
$\theta_0^n$-closed subset of $X$ and $\mathcal{U}$ is an open
cover of $A$, then there is $\mathcal{V}\subset \mathcal{U}$ with
$|\mathcal{V}|\leq \kappa$ and $A\subset\overline{\bigcup
\mathcal{V}}$.
\end{definition}

\begin{definition} For a space $X$ and $n\in \mathbb{N}$, let $\kappa_{\theta_0(n)}(X)$ be the smallest cardinal $\kappa$ such
that for each point $x\in X$, there is a collection
$\mathcal{V}_{x}$ of $n$-hulls of $x$ so that
$|\mathcal{V}_{x}|\leq \kappa$ and if $W$ is a $n$-hull of $x$,
then $W$ contains a member of $\mathcal{V}_{x}$.
\end{definition}

\begin{lemma} For a subset $A$ of an $S(2n-1)$-space $X$,
$|cl_{\theta_0 ^n}{A}|\leq|A|^{\kappa_{\theta_0 (n)}(X)}$.
\end{lemma}

\begin{theorem}\label{th11}  If $X$ is an $S(2n-1)$-space, then
$|X|\leq 2^{sL_{\theta_0 (n)}(X)\kappa_{\theta_0 (n)}(X)}$.
\end{theorem}

\begin{definition} For a space $X$, the $\theta_0(n)$-quasi-Menger
number $qM_{\theta_0(n)}(X)$ is the smallest cardinal number
$\kappa$ such that for every closed subset $A$ of $X$ and every
collection $\{\mathcal{U}_{\alpha}: \alpha\leq \kappa \}$ of
families of open subsets of $X$ with $A\subset
\bigcup_{\alpha<\kappa} (\bigcup \mathcal{U}_{\alpha})$, there are
finite subfamilies $\mathcal{V}_{\alpha}$ of
$\mathcal{U}_{\alpha}$, $\alpha<\kappa$, such that $A\subset
\bigcup_{\alpha<\kappa} cl_{\theta_0^n}(\bigcup
\mathcal{V}_{\alpha})$.

\end{definition}

\begin{theorem}\label{th33} For every $S(2n-1)$-space $X$, $|X|\leq 2^{qM_{\theta_0
(n)}(X)\kappa_{\theta_0 (n)}(X)}$.
\end{theorem}

\bigskip

 In \cite{str}, L. Stramaccia defined the notion of $S(n)$-set.

\begin{definition} (\cite{str}) Let $X$ be a topological space and
$M$ be a subset of $X$.

$\bullet$ A cover $\mathcal{U}=\{U_{\alpha}: \alpha\in \Lambda\}$
of $M$ by open sets of $X$, is an
 {\it $S(n)$-cover with respect to $M$}, if
 $M\subset\bigcup\{Int_{\theta^n}U_{\alpha}: \alpha\in \Lambda\}$.

$\bullet$ $M$ is an  {\it $S(n)$-set} of $X$ if every $S(n)$-cover
with respect to $M$ has a finite subcover.

\end{definition}

\begin{definition} For a space $X$ and $n\in \mathbb{N}$, {\it $sL_{s(n)}(X)$}
is the smallest cardinal $\kappa$ such that if $A$ is a
$\theta^n$-closed subset of $X$ and $\mathcal{U}$ is an
$S(n)$-cover with respect to $A$, then there is
$\mathcal{V}\subset \mathcal{U}$ with $|\mathcal{V}|\leq \kappa$
and $A\subset\overline{\bigcup \mathcal{V}}$.
\end{definition}

\bigskip

Note that $sL_{s(n)}(X)\leq sL_{\theta(n)}(X)$ for every $n\in
\mathbb{N}$.

\bigskip
\begin{theorem}\label{th5} For every $S(2n)$-space $X$,
 $|X|\leq
2^{sL_{s(n)}(X)\kappa_{\theta (n)}(X)}$.
\end{theorem}

\begin{proof}
Note that in the proof of Theorem \ref{th1}, $\{Int(D_{x}):x\in
A\}$ is an $S(n)$-cover with respect to $A$ and
$A=cl_{\theta^n}(A)$, hence, in the same way as in Theorem
\ref{th1}, we obtain a complete proof.
\end{proof}





In paper \cite{gk}, for a topological space $X$ and $n\in
\mathbb{N}$, Gotchev and Ko$\check{c}$inac define the interesting
cardinal functions $d_n(X)$ and $bt_n(X)$, called respectively
$S(n)$-density and $S(n)$-bitightness, and using them authors
prove two representative results: if $X$ is an $S(n)$-space, then
$|X|\leq 2^{2^{d_n(X)}}$ and $|X|\leq [d_n(X)]^{bt_n(X)}$.

\begin{remark} In paper \cite{gk}, the authors use the notation
$\chi_{2n}(X)$ instead of the author's $\kappa_{\theta(n)}(X)$ and
$\chi_{2n-1}(X)$ instead of the author's $\kappa_{\theta_0(n)}(X)$
(see Definition 2.7 (b) and (c) in \cite{gk}).
\end{remark}

 The paper \cite{got} also contains some results about cardinal
 inequalities for $S(n)$-spaces, extending some classical results
 of Hajnal and Juh$\acute{a}$sz, and Schr$\ddot{o}$der.

{\bf Acknowledgment.} I should like to Thanks to the anonymous
referee who read carefully the manuscript and helped my to
simplify and improve the presentation of results of the paper.
\bigskip





\bibliographystyle{model1a-num-names}
\bibliography{<your-bib-database>}

\begin{thebibliography}{10}

\bibitem{al}
O.T. Alas, \textit{More topological cardinal inequalities},
Colloq. Math. {\bf 65}, (1993), 165--168.


\bibitem{alko}
O.T. Alas, Lj.D.R. Ko$\check{c}$inac, \textit{More cardinal
inequalities on Urysohn spaces}, Math.Balkanica -- new series {\bf
14} (2000), no.3-4., 247--251.


\bibitem{arh79}
A. V. Arhangel'skii, \textit{A theorem on cardinality}, Russian
Math. Surveys, {\bf 34}:4, (1979), 153--154.



\bibitem{arh}
A. V. Arhangel'skii, \textit{A generic theorem in the theory of
cardinal invariants of topological spaces}, Comment. Math. Univ.
Carolinae, {\bf 36}:2, (1995), 303-–325.

\bibitem{bbc}
F.A. Basile, M. Bonanzinga and N. Carlson, \textit{Variations on
known and recent cardinality bounds}, Topology Appl., {\bf 240},
(2018), 228--237.


\bibitem{beca}
A. Bella, F. Cammaroto, \textit{On the cardinality of Urysohn
spaces}, Canad. Math. Bull. Vol. {\bf 31}:2, (1988), 153--158.


\bibitem{bc}
A. Bella, N. Carlson, \textit{On cardinality bounds involving the
weak Lindel$\ddot{o}$f degree}, Quaestionees Mathematicae, {\bf
41}:1, (2018), 99--113.


\bibitem{ccp}
F. Cammaroto, A. Catalioto and J. Porter, \textit{On the
cardinality of Urysohn spaces}, Topology Appl., {\bf 160}:14,
(2013), 1862--1869.




\bibitem{digi}
D. Dikranjan and E. Giuli, \textit{$S(n)$-$\theta$-closed spaces},
Topology Appl., {\bf 28}, (1988), 59–-74.



\bibitem{got}
I.S. Gotchev, \textit{Cardinal inequalities for $S(n)$-spaces},
Acta Math. Hungar. (2019),
https://doi.org/10.1007/s10474-019-00939-0.


\bibitem{gk}
I.S. Gotchev, Lj.D.R. Ko$\check{c}$inac, \textit{More on the
cardinality of $S(n)$-spaces}, Serdica Math. J. {\bf 44} (1-2),
(2018), 227--242.



\bibitem{gr}
A. A. Gryzlov, \textit{$H$-closed spaces and compactness-type
properties}, Candidate's Dissertation in Physics and Mathematics,
Sverdlovsk (1973).


\bibitem{ham}
T. Hamlett, \textit{$H$-closed spaces and the associated
$\theta$-convergence space}, Math. Chronicle, {\bf 8}, (1979),
83-–88.

\bibitem{her}
H. Herrlich, \textit{$T_{\theta}$-Abgeschlossenheit und
$T_{\theta}$-Minimalitat}, Math. Z., {\bf 88}, (1965), 285–-294.

\bibitem{hod}
R.E. Hodel, \textit{Arhangel'skii's solution to Alexandroff's
problem}, Topol. Appl. {\bf 153}, (2006), 2199--2217.



\bibitem{jan}
D. Jankovic, "On some separation axioms and $\theta$-closure,"
Mat. Vesnik, {\bf 4}:17, (1980), 439–-449.

\bibitem{jrw}
S. Jiang, I. Reilly, and S. Wang, \textit{Some properties of
$S(n)$-$\theta$-closed spaces}, Topology Appl., {\bf 96}, (1999),
23-–29.



\bibitem{ko}
Lj.D.R. Ko$\check{c}$inac, \textit{On the cardinality of Urysohn
spaces, Q $\&$ A}, General Topology, {\bf 13}:2, (1995), 211--216.





\bibitem{os1}
A. V. Osipov, \textit{An example of a nonfeebly compact product of
$U$-$\theta$-closed spaces}, Proc. Steklov Inst. Math., {\bf 2},
(2001), 186–-188.

\bibitem{os2}
A. V. Osipov, \textit{Different kinds of closedness in
$S(n)$-spaces}, Proc. Steklov Inst. Math., {\bf 1}, (2003),
155–-160.

\bibitem{os3}
A. V. Osipov, \textit{Weakly $H$-closed spaces}, Proc. Steklov
Inst. Math., {\bf 1}, (2004), 15–-17.

\bibitem{os4}
A. V. Osipov, \textit{Nearly $H$-closed spaces}, Journal of
Mathematical Sciences, {\bf 155}:4, (2008), 626--633.


\bibitem{os5}
A. V. Osipov, \textit{Some properties of minimal $S(\alpha)$ and
$S(\alpha)FC$ spaces}, Topology Proceedings, {\bf 50}, (2017),
79--86.

\bibitem{pol}
R. Pol, \textit{Short proofs of two theorems on cardinality of
topological spaces}, Bull. Acad. Pol. Sci. Ser. Math., {\bf 22},
(1974), 1245--1249.


\bibitem{por}
J.R. Porter, \textit{Minimal $R(\omega_0)$ spaces}, Notices Amer.
Math. Soc. {\bf 16} (1969), 218.

\bibitem{pr2}
J.R. Porter, \textit{Categorical problems in minimal spaces}, in:
Proc. First Categorical Topology Symposium, Lecture Notes in Math.
540 (Springer, Berlin, 1976), 483--500.



\bibitem{povo}
J.R. Porter and C. Votaw, \textit{$S(\alpha)$-spaces and regular
Hausdorff extensions}, Pac. J. Math., {\bf 45}, (1973), 327-–345.

\bibitem{str}
L. Stramaccia,  \textit{$S(n)$-spaces and $H$-sets}, Comment.
Math. Univ. Carolinae, {\bf 29}:2, (1988), 221-–226.

\bibitem{vel}
N. V. Veli$\check{c}$ko, \textit{$H$-closed topological spaces},
Mat. Sb. (N.S.), {\bf 70}: 112, (1966), 98–-112.

\bibitem{vig}
G. Viglino, \textit{$\overline{T}_n$-spaces}, Kyungpook Math. J.
{\bf 11}, (1971), 33--35.




\end{thebibliography}



\end{document}